\documentclass[10pt,a4paper]{article}
\usepackage{amsmath,amsthm,amssymb,amscd}
\usepackage{graphicx,graphics}
\usepackage{verbatim}
\usepackage{mathrsfs}
\usepackage{enumerate}
\usepackage{color}
\usepackage[top=1.2in, bottom=1.2in, left=1.2in, right=1.2in]{geometry}

\newtheorem{theorem}{Theorem}[section]

\newtheorem{problem}[theorem]{Problem}
\newtheorem{lemma}[theorem]{Lemma}

\newtheorem{claim}[theorem]{Claim}

\begin{document}

\title{The 3-colorability of planar graphs without cycles of length 4, 6 and 9}

\author{Yingli Kang\thanks{
        Fellow of the International Graduate
        School "Dynamic Intelligent Systems";
        Paderborn Institute for Advanced Studies in
		Computer Science and Engineering,
		Paderborn University,
		33102 Paderborn,
		Germany.},
       Ligang Jin\thanks{
       Corresponding author; Supported by Deutsche Forschungsgemeinschaft (DFG) grant STE 792/2-1;
       Paderborn University.}
       Yingqian Wang\thanks{
       supported by National Natural Science Foundation of China (NSFC) 11271335;
Department of Mathematics, Zhejiang Normal
University, 321004 Jinhua, China; 
yingli@mail.upb.de, ligang@mail.upb.de, yqwang@zjnu.cn.}
}

\maketitle

\begin{abstract}
\small In this paper, we prove that planar graphs without cycles of length 4, 6, 9 are 3-colorable.
\end{abstract}

\section{Introduction}
 The well-known Four Color Theorem
states that every planar graph is 4-colorable. On the 3-colorability
of planar graphs, a famous theorem owing to Gr\"{o}tzsch \cite{Grotzsch1959109} states
that every planar graph without cycles of length 3 is 3-colorable.
Therefore, next sufficient conditions that guarantee 3-colorability of planar graphs should always allow the presence of cycles of length 3.
In 1976, Steinberg conjectured that every planar graph without cycles
of length 4 and 5 is 3-colorable. Erd\"{o}s \cite{Steinberg1993211} suggested a relaxation of Steinberg's Conjecture: does there exist a constant $k$ such
that every planar graph without cycles of length from 4 to $k$ is
3-colorable? Abbott and Zhou \cite{AbbottZhou1991203} proved that such a constant exists and
$k\leq 11$. This result was later on improved to $k\leq 9$ by
Borodin \cite{Borodin1996183} and, independently, Sanders and Zhao \cite{SandersZhao199591}, and to
$k\leq 7$ by Borodin, Glebov, Raspaud and Salavatipour \cite{BorodinEtc2005303}.
Besides, much attention was paid to sufficient conditions that forbid cycles of some other certain length.
The results concerning four kinds of forbidden length of cycles were obtained in several different papers and summarized in \cite{LuEtc20094596}:
\begin{theorem}
A planar graph is 3-colorable if it
has no cycle of length $4$, $i$, $j$ and $k$, where $5\leq i<j<k\leq
9$.
\end{theorem}

A more general problem than Steinberg's was formulated also in \cite{LuEtc20094596}:
\begin{problem} \label{pro_forbidden 4 and i}
What is $\cal{A}$, a set of integers between 5 and 9, such that for $i\in \cal{A}$, every planar graph with cycles
of length neither 4 nor $i$ is 3-colorable?
\end{problem}

It seems very far to settle Problem \ref{pro_forbidden 4 and i}, since no element of such a set $\cal{A}$ is found.
Therefore, a reasonable way to deal with this problem is to ask following question:

\begin{problem} \label{pro_forbidden 4 i and j}
What is $\cal{B}$, a set of pairs of integers
$(i,j)$ with $5\leq i<j\leq 9$, such that planar graphs without
cycles of length 4, $i$ and $j$ are 3-colorable?
\end{problem}

The first step towards Problem \ref{pro_forbidden 4 i and j} was made by
Xu \cite{Xu2006958}, who proved that a planar graph is 3-colorable if it has neither 5- and 7-cycles nor adjacent 3-cycles. Unfortunately, there is a gap in his proof, as pointed out by Borodin etc. \cite{BorodinEtc2009668}, who later on gave a new proof of the same statement. Afterwards, Xu \cite{Xu2009347} fixed this gap. Hence $(5,7)\in \cal{B}$.
Other known elements of $\cal{B}$ includes pair (6,8) given by Wang and Chen \cite{WangChen20071552}, pair (7,9) given by Lu etc. \cite{LuEtc20094596}, and pair (6,7) given by Borodin, Glebov and Raspaud \cite{BorodinEtc20102584}. Actually, the theorem proved in \cite{BorodinEtc20102584} states that planar graphs without triangles adjacent to cycles of length from 4 to 7 are 3-colorable, which implies $(6,7)\in \cal{B}$.

%Recently, as a relaxation of proper coloring, improper coloring were frequently discussed related to Problems \ref{pro_forbidden 4 and i} and \ref{pro_forbidden 4 i and j}. Definition,  3-colorable = (0,0,0)-colorable, result

In this paper, we show that $(6,9)\in \cal{B}$, that is, we prove the following theorem:

\begin{theorem} \label{thm469}
Every planar graph without cycles of length 4, 6, 9 is 3-colorable.
\end{theorem}

The graphs considered in this paper are finite and simple.
Let $G$ be a plane graph and $C$ a cycle of $G$.
By $Int(C)$ (or $Ext(C)$) we denote the subgraph of $G$ induced by the vertices lying inside (or outside) of $C$.
Cycle $C$ is \textit{separating} if both $Int(C)$ and $Ext(C)$ are not empty.
By $\overline{Int}(C)$ (or $\overline{Ext}(C)$) we denote the subgraph of $G$ consisting of $C$ and its interior (or exterior).
%By $int(C)$ and $ext(C)$ we denote the set of vertices lying inside and outside of $C$, respectively.
%$C$ is separating if both $int(C)$ and $ext(C)$ are not empty.
%Define $Int(C)=G-ext(C)$ and $Ext(C)=G-int(C)$.

Denote by $G[S]$ the subgraph of $G$ induced by $S$, where either $S\subseteq V(G)$ or $S\subseteq E(G)$.
A vertex is a \textit{neighbor} of another vertex if they are adjacent.
A \textit{chord} of $C$ is an edge of $\overline{Int}(C)$ that connects two nonconsecutive vertices on $C$.
If $Int(C)$ has a vertex $v$ with three neighbors $v_1,v_2,v_3$ on $C$, then $G[\{vv_1,vv_2, vv_3\}]$ is called a \textit{claw} of $C$.
If $Int(C)$ has two adjacent vertices $u$ and $v$ such that $u$ has two neighbors $u_1,u_2$ on $C$ and $v$ has two neighbors $v_1,v_2$ on $C$, then $G[\{uv,uu_1,uu_2,vv_1,vv_2\}]$ is called a \textit{biclaw} of $C$.
If $Int(C)$ has three pairwise adjacent vertices $u,v,w$ such that $u,v$ and $w$ have a neighbor $u',v'$ and $w'$ on $C$ respectively, then $G[\{uv,vw,uw,uu',vv',ww'\}]$ is called a \textit{triclaw} of $C$ (see Figure \ref{fig_claw}).

\begin{figure}[h]
  \centering
  \includegraphics[width=13cm]{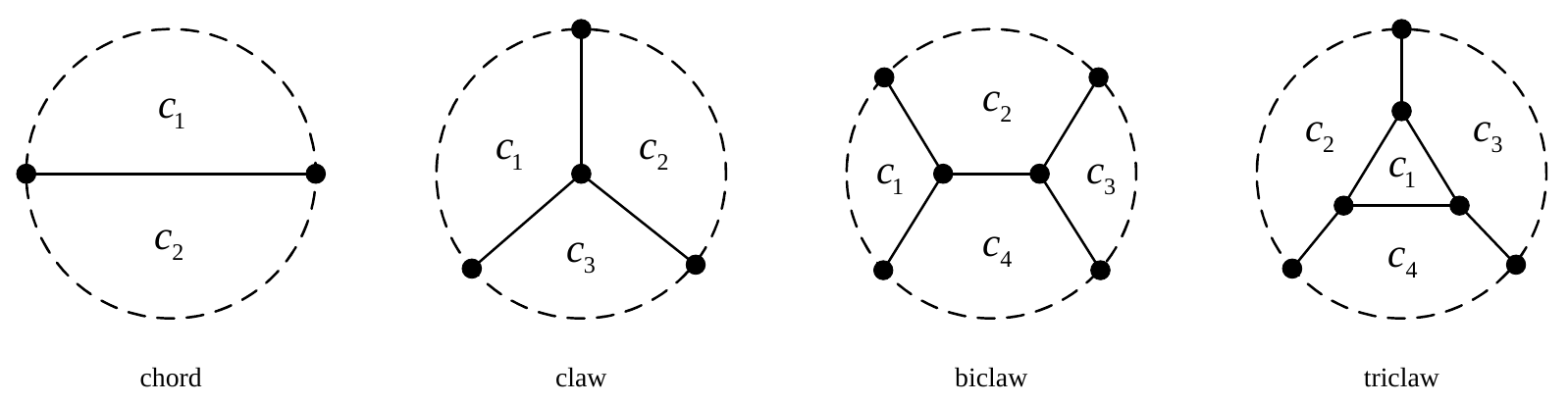}\\
  \caption{chord, claw, biclaw and triclaw of a cycle}\label{fig_claw}
\end{figure}
Let $C$ be a cycle and $T$ be one of the chords, claws, biclaws and triclaws of $C$. We call the graph consisting of $C$ and $T$ a \textit{bad partition} $H$ of $C$. The boundary of any one of the parts, into which $C$ is divided by $H$, is called a \textit{cell} of $H$. Clearly, every cell is a cycle.
In case of confusion, let us always order the cells $c_1,\cdots,c_t$ of $H$ in the way as shown in Figure \ref{fig_claw}.
For every cell $c_i$ of $H$, let $k_i$ be the length of $c_i$. Then $T$ is further called a $(k_1,k_2)$-chord, a $(k_1,k_2,k_3)$-claw, a $(k_1,k_2,k_3,k_4)$-biclaw or a $(k_1,k_2,k_3,k_4)$-triclaw, respectively.

%A $(k_1,k_2)$-chord of a cycle is a chord to whom the cells respect are of lengths sequence $k_1,k_2$. Similarly,
%a $(k_1,k_2,k_3)$-claw of a cycle is a claw to whom the cells respect are of lengths sequence $k_1,k_2,k_3$;
%a $(k_1,k_2,k_3,k_4)$-biclaw (or -triclaw) of a cycle is a biclaw (or triclaw) to whom the cells respect are of lengths sequence $k_1,k_2,k_3,k_4$.
Let $k$ be a positive integer.
A $k$-cycle is a cycle of length $k$. A $k^-$-cycle (or $k^+$-cycle) is a cycle of length at least (or at most) $k$.
A \textit{good cycle} is a 12$^-$-cycle that has none of claws, biclaws and triclaws.
A \textit{bad cycle} is a 12$^-$-cycle that is not good.
We say a 9-cycle is \textit{special} if it has a (3,8)-chord or a (5,5,5)-claw.
%The notions of bad cycles and special 9-cycles are crucial to our proof.

Let $\cal{G}$ be the class of connected plane graphs with neither 4- and 6-cycle nor special 9-cycle.

Instead of Theorem \ref{thm469}, it is easier for us to prove the following stronger one:
\begin{theorem} \label{thm46s9}
Let $G\in \cal{G}$. We have
\begin{enumerate}[(1)]
  \item $G$ is 3-colorable; and
  \item If $D$, the boundary of the exterior face of $G$, is a good cycle, then every proper 3-coloring of $G[V(D)]$ can be extended to a proper 3-coloring of $G$.
\end{enumerate}
\end{theorem}
%A claw of $C$ is a subgraph of $G$ induced by a vertex inside $C$ together with its three neighbors in $C$.
%A biclaw of $C$ is a subgraph of $G$ induced by two adjacent vertices, say $u, v$, inside $C$ together with two neighbors in $C$ of each of $u, v$. A triclaw of $C$ is a %subgraph of $G$ induced by three mutual adjacent vertices, say $u, v, w$, inside $C$ together with one neighbor in $C$ of each of $u, v ,w$.

This section is concluded with some notations that are used in the next section. Let $G$ be a plane graph.
Denote by $d(v)$ the degree of a vertex $v$, by $|C|$ the length of a cycle $C$ and by $|f|$ the size of a face $f$. Let $k$ be a positive integer.
A \textit{$k$-vertex} is a vertex of degree $k$, and a \textit{$k$-face} is a face of size $k$. A \textit{$k^+$-vertex} (or \textit{$k^-$-vertex}) is a vertex of degree at least (or at most) $k$, and a \textit{$k^+$-face} (or \textit{$k^-$-face}) is a face of size at least (or at most) $k$.
A \emph{$k$-path} is a path that contains $k$ edges.
A $k$-cycle containing vertices $v_1,\ldots,v_k$ in cyclic order is denoted by $[v_1\ldots v_k]$.
Denote by $N(v)$ the set of neighbors of a vertex $v$.
Let $N_H(v)=N(v)\cap V(H)$ whenever $v$ is a vertex of a subgraph $H$ of $G$.
A vertex is \textit{external} if it lies on the exterior face, \textit{internal} otherwise.
A vertex incident with a triangle is called a \textit{triangular vertex}.
We say a vertex is \textit{bad} if it is an internal triangular 3-vertex; \textit{good} otherwise.
A path is a \textit{splitting path} of a cycle $C$ if it has two end-vertices on $C$ and all other vertices inside $C$.
We say a path is \textit{good} if it contains only internal 3-vertices and has an end-edge incident with a triangle.
A cycle or a face $C$ is \textit{triangular} if $C$ is adjacent to a triangle $T$. Furthermore, if $C$ is a cycle and $T\in \overline{Ext}(C)$, then we say $C$ is an \textit{ext-triangular} cycle. A triangular 7-face is \textit{light} if it has no external vertex and every incident nontriangular vertex has degree 3.
\section{Proof of Theorem \ref{thm46s9}}
Suppose to the contrary that Theorem \ref{thm46s9} is false.
From now on, let $G$ be a counterexample to Theorem \ref{thm46s9} with fewest vertices.
Actually, $G$ violates the second conclusion of Theorem \ref{thm46s9}, since conclusion (2) implies conclusion (1).
We still use $D$ to denote the boundary of the exterior face of $G$, and let $\phi$ be a proper 3-coloring of $G[V(D)]$ which cannot be extended to a proper 3-coloring of $G$. Clearly, $D$ is a good cycle. By the minimality of $G$, $D$ has no chord.

%The next two subsections are devoted to structural properties of $G$ and discharging in $G$, respectively.
%A contradiction with Euler's formula is finally obtained in section \ref{secch}, which completes the proof of Theorem \ref{thm46s9}.

\subsection{Structural properties of minimal counterexample $G$}
\begin{lemma} \label{lem_min degree}
Every internal vertex of $G$ has degree at least 3.
\end{lemma}

\begin{proof}
 Suppose to the contrary that $G$ has an internal vertex $v$ such that $d(v)\leq 2$. We can extend $\phi$ to $G-v$ by the minimality of $G$, and then to $G$ by coloring $v$ different from its neighbors.
\end{proof}

\begin{lemma}
$G$ is 2-connected and therefore, the boundary of each face of $G$ is a cycle.
\end{lemma}

\begin{proof}
Otherwise, we may assume that $G$ has a pendant block $B$ with cut vertex $v$ such that $B-v$ does not intersect with $D$.
We first extend $\phi$ to $G-(B-v)$, and then 3-color $B$ such that the color assigned to $v$ is unchanged.
\end{proof}

\begin{lemma} \label{lem_sep good cycle}
$G$ has no separating good cycle.
\end{lemma}

\begin{proof}
Suppose to the contrary that $G$ has a separating good cycle $C$. We extend $\phi$ to $G-Int(C)$.
Furthermore, since $C$ is a good cycle, the color of $C$ can be extended to its interior.
\end{proof}

%We show some needed facts on the bad cycles of $G$, which can be easily obtained by lemma 3 and condition $G\in \cal{G}$.
One can easily conclude following three lemmas.

\begin{lemma} \label{lem_cycle of G}
Every $9^-$-cycle of $G$ is facial except that an 8-cycle of $G$ might have a chord, which is a (3,7)- or (5,5)-chord.
\end{lemma}

\begin{lemma} \label{bad cycle}
Let $H\in \cal{G}$. If $C$ is a bad cycle of $H$, then $C$ has length either 11 or 12. Furthermore, if $|C|=11$, then $C$ has a (3,7,7)- or (5,5,7)-claw; if $|C|=12$, then $C$ has a (5,5,8)-claw, a (3,7,5,7)- or (5,5,5,7)-biclaw, or a (3,7,7,7)-triclaw.
\end{lemma}

\begin{lemma} \label{lem triangular bad cycle}
Every bad cycle $C$ of $G$ is adjacent to at most one triangle. Furthermore, if $C$ is ext-triangular, then $C$ has either a (5,5,7)-claw or a (5,5,5,7)-biclaw.
%(3)~$uv\in E(H)$ for every $u,v\in int(C)$
\end{lemma}

% (1) ~If $C$ is a bad cycle of a graph from $\cal{G}$, then $C$ has length 11 or 12. Furthermore, if $|C|=11$, then $C$ has a (3,7,7)- or (5,5,7)-claw (see fig ?); if $|C|=12$, then $C$ has a (5,5,8)-claw, or a (3,5,7,7)- or (5,5,5,7)-biclaw, or a (3,7,7,7)-triclaw (see fig ?).

% (2) ~If $C$ is a bad cycle of a graph from $\cal{G}$, then $C$ is adjacent to at most one triangle outside. Furthermore, if $C$ is adjacent precisely one triangle outside, then $C$ has either a (5,5,7)-claw or a (5,5,5,7)-biclaw.

% (3) ~If $C$ is a $12^-$-cycle of $G$ and $x, y$ are two distinct vertices inside $C$, then $C$ is a bad 12-cycle and $x$ is adjacent to $y$.

% (4) ~Any $9^-$-cycle of $G$ is facial except that an 8-cycle might have one chord which is either a (3,7)- or a (5,5)-chord.

\begin{lemma}\label{lem_splitting path}
Let $P$ be a splitting path of $D$ which divides $D$ into two cycles $D'$ and $D''$.
\begin{enumerate}[(1)]
  \item If $|P|=2$, then there is a 3-face between $D'$ and $D''$;
  \item If $|P|=3$, then there is a 5-face between $D'$ and $D''$;
  \item If $|P|=4$, then there is a 5- or 7-face between $D'$ and $D''$;
  \item If $|P|=5$, then there is a $9^-$-cycle between $D'$ and $D''$.
\end{enumerate}
\end{lemma}

\begin{proof}
Since $D$ has length at most 12, we have $|D'|+|D''|=|D|+2|P|\leq 12+2|P|$.
Recall that every $7^-$-cycle of $G$ is a facial cycle by Lemma \ref{lem_cycle of G}.

(1) ~Let $P=xyz$. Suppose to the contrary that $|D'|, |D''| \geq 5$.
By Lemma \ref{lem_min degree}, $y$ has a neighbor other than $x$ and $z$, say $y'$. It follows that $y'$ is internal since otherwise $D$ is a bad cycle with a claw. Without loss of generality, let $y'$ lie inside $D'$. Thus $|D'|\geq 11$ by Lemma \ref{lem_sep good cycle}. Since $|D'|+|D''|\leq 16$, we have $|D'|=11$ and $|D''|=5$. Now $D'$ has a claw by Lemma \ref{bad cycle}, which implies that $D$ has a biclaw, a contradiction.

(2) ~Let $P=wxyz$. Suppose to the contrary that $|D'|, |D''| \geq 7$.
Let $x'$ and $y'$ be a neighbor of $x$ and $y$ not on $P$, respectively.
If both $x'$ and $y'$ are external, then $D$ has a biclaw.
Hence, we may assume $x'$ lies inside $D'$.
By Lemma \ref{bad cycle} and inequality $|D'|+|D''|\leq 18$, we have $|D'|=11$ and $|D''|$=7.
Thus $D'$ has a claw which divides $D'$ into three faces. Since $D''$ is facial, $y'$ can only coincide with $x'$. Now $D$ has a triclaw.

(3) ~Let $P=vwxyz$. Suppose to the contrary that $|D'|, |D''| \geq 8$.
Since $|D'|+|D''|\leq 20$, we have $|D'|, |D''| \leq 12$.
If $G$ has an edge $e$ connecting two nonconsecutive vertices on $P$, then $e$ together with $P$ can form only a triangle.
Without loss of generality, let $e=wy$ and $e$ belongs to $\overline{Int}(D')$. Now path $vwyz$ is a splitting 3-path of $D$ and hence $D'$ is a 6-cycle with a (3,5)-chord, a contradiction.
Therefore, no pair of nonconsecutive vertices on $P$ are adjacent.
%This will be used without mention in the followings.

Let $w', x', y'$ be a neighbor of $w, x, y$ not on $P$, respectively.
If $x'$ is external, say $xx'$ is a chord of $D'$, then both of paths $vwxx'$ and $x'xyz$ are splitting 3-paths of $D$. It follows that $D'$ is an 8-cycle with a (5,5)-chord $xx'$.
Hence $y'$ has no other possibility but to lie inside of $D''$, and so does $w'$.
By noticing that $w'$ cannot coincide with $y'$, we know $D''$ is a bad 12-cycle. It follows that $G$ has an edge connecting $w'$ and $y'$, which yields a special 9-cycle of $G$.
Therefore, vertex $x'$ is internal.

We may assume $x'$ lies inside of $D'$.
Thus $D'$ is a bad 11- or 12-cycle, which implies $D''$ has length 8 or 9.
If $|D''|=9$, then $D''$ is facial and $D'$ is a bad 11-cycle with a claw, which is impossible because of the locations of $w'$ and $y'$. Hence we may assume $|D''|=8$. It follows that not both $w'$ and $y'$ lie in $\overline{Int}(D'')$ and that $w',x',y'$ are pairwise distinct. Now $G$ has a 4-cycle that is either $[wxx'w']$ or $[xyy'x']$, a contradiction.

(4) ~Let $P=uvwxyz$. Suppose to the contrary that $|D'|, |D''| \geq 10$.
Since $|D'|+|D''|\leq 22$, we have $|D'|, |D''| \leq 12$.
By similar argument as in (3), one can conclude that $G$ has no edge connecting two nonconsecutive vertices on $P$.
Let $v', w', x', y'$ be a neighbor of $v, w, x, y$ not on $P$, respectively.

We claim that both vertices $w'$ and $x'$ are internal. Otherwise, let $ww'$ be a chord of $D'$. Since both $uvww'$ and $w'wxyz$ are splitting paths of $D$, $D'$ is a 10-cycle with a (5,7)-chord $ww'$. Thus all of $v',x'$ and $y'$ belong to $\overline{Int}(D'')$. If $x'$ is external, then similarly, $D''$ is a 10-cycle with a (5,7)-chord $xx'$, which is impossible because of the location of $y'$. Hence, we may assume that $x'$ lies inside $D''$. Furthermore, $v'$ also lies inside $D''$, since otherwise $G$ has 3-face $[uvv']$ adjacent to a 5-face. Clearly, $v'\neq x'$. Hence $D''$ is a bad 12-cycle containing two adjacent vertices $v'$ and $x'$ inside. A contradiction is obtained by noticing both the location of $y'$ and the specific interior of $D''$.
%If $D''$ has a biclaw, then we have $y'=x'$ for otherwise $y$ is incident with two faces. Now $G$ has a 6-cycle $[vwxyx'v'v]$. If $D''$ has a triclaw, then $G$ has a 6-cycle with chord $x'v'$.

Let $w'$ lie inside $D'$. If $x'$ lies inside $D''$, then both $D'$ and $D''$ are bad 11-cycles.
It follows that $v'=w', y'=x'$, and both $D'$ and $D''$ have a $(3,7,7)$-claw, yielding special 9-cycles of $G$.
%The special structure of bad 11-cycle gives that $v'=w', y'=x'$ and thus the claw of $D'$ splits $D'$ into three faces %with degree sequence 3,7,7, and so does the claw of $D''$.
%Now $G$ has a special 9-cycle with (3,8)-chord $vw$.
Hence, we may assume that $x'$ lies also inside $D'$. It follows that $x'$ coincide with $w'$, since otherwise the adjacency of $x'$ and $w'$ gives a 4-cycle of $G$.
Thus $D'$ is a bad cycle with either a $(3,7,7)$-claw or a (3,7,5,7)-biclaw, which implies both $v'$ and $y'$ belong to $\overline{Int}(D'')$.
%It follows that $D'$ is either a bad 11-cycle with a claw splitting $D'$ into three faces of degree sequence 3,7,7 or a %bad 12-cycle with a biclaw splitting $D''$ into four faces of degree sequence 3,5,7,7.
%Thus we have $v', y' \in Int(D'')$.
If $v'$ lies on $D''$, then $G$ has triangle $[uvv']$ adjacent to an 8-cycle of $D'$, a contradiction.
%Moreover, $v', y' \in int(D'')$, for otherwise one of paths $uvv'$ and $zyy'$ is a splitting 2-path towards $D$ which together with $D$ forms a triangle adjacent to a 8-cycle of $D'$.
Hence we may assume $v'$ lies inside $D''$ and so does $y'$.
It follows that either $v'= y'$ or $v'y'\in E(G)$, yielding 6-cycles of $G$ in both cases.
%If $v'= y'$, then $G$ has a 6-cycle $[vww'xyy'v]$. Hence we have $v' \neq y'$. It follows that $G$ has a 6-cycle %$[vwxyy'v'v]$, a contradiction.

The proof of this lemma is completed.
\end{proof}

\begin{lemma} \label{operation}
Let $G'$ be a plane graph obtained from $G$ by a graph operation $T$.

Let $T$ consist of deleting a nonempty set of internal vertices and either identifying two vertices or
adding an edge between two nonadjacent vertices.
%inserting an edge into a face.
If after $T$ we\\
$(a)$~ identify no two vertices on $D$, and create no edge connecting two vertices on $D$, and\\
$(b)$~ create neither $6^-$-cycle nor ext-triangular 7- or 8-cycle,\\
then $\phi$ can be extended to $G'$.

Let $T$ consist of deleting a nonempty set $S$ of internal vertices and identifying
two edges $u_1u_2$ and $v_1v_2$ so that $u_1$ is identified with $v_1$.
For $i\in \{1,2\}$, let $T_i$ denote the operation on $G$ that consists of deleting all vertices in $S$ and identifying $u_i$ and $v_i$.
If at least one of $u_1u_2$ and $v_1v_2$ is contained in no $8^-$-cycle of $G-S$, and if conditions $(a)$ and $(b)$ above hold for both $T_1$ and $T_2$,
then $\phi$ can be extended to $G'$.
\end{lemma}

\begin{proof}
First let $T$ consist of deleting a nonempty set of internal vertices and identifying two other vertices $t_1$ and $t_2$.
Let $t'$ denote the vertex obtained from $t_1$ and $t_2$ after $T$.
Conditions $(a)$ and $(b)$ implies (i) to show $G'\in \cal{G}$, it suffices to show $G'$ has no special 9-cycles; and (ii) $D$ bounds $G'$ and $\phi$ is a proper 3-coloring of $G'[V(D)]$.
Therefore, $\phi$ can be extended to $G'$ by the minimality of $G$ if we can show both that $G'$ has no special 9-cycles and that $D$ is good in $G'$.

Suppose $G'$ has a special 9-cycle $C$. Let $H$ be a bad partition of $C$. We have $t'\in V(H)$ since otherwise $C$ is a special 9-cycle in $G$. Condition $(b)$ implies that every vertex of $N_H(t')$ is adjacent to precisely one of $t_1$ and $t_2$ in $G$.
If all the vertices of $N_H(t')$ is adjacent to $t_1$, then $C$ is a special 9-cycle in $G$.
Hence, we may assume that $N_H(t')$ has a vertex adjacent to $t_2$ and similarly, has another vertex adjacent to $t_1$.
%If $t_1$ is adjacent to none of $N_H(t')$, then $t_2$ is adjacent to all vertices of $N_H(t')$ in $G$. It follows that $C$ is a special 9-cycle in $G$.
%Hence, we may assume that $t_1$ has a neighbor in $N_H(t')$, and so does $t_2$.
Thus after $T$ a cell of $H$ containing $t'$ is created, that is, we have created a 3- or 5-cycle or an ext-triangular 8-cycle, contradicting $(b)$. Therefore, $G'$ has no special 9-cycle.

Suppose $D$ is bad in $G'$. Let $H$ be a bad partition of $D$. We have $t'\in V(H)$ since otherwise $D$ is bad in $G$.
If $t'$ has degree 2 in $H$, then $t_1, t_2 \in V(D)$ since otherwise $D$ is bad in $G$.
Now we identify two vertices on $D$, contradicting $(a)$.
Hence $t'$ has degree 3 in $H$. Similarly as paragraph above, we may assume that $N_H(t')$ has a vertex $w_1$ adjacent to $t_1$ and two other vertices $w'_2, w''_2$ adjacent to $t_2$ in $G$. It follows that $H$ has two cells containing either $w_1t'w'_2$ or $w_1t'w''_2$ created by $T$.
Clearly, $G' \in \cal{G}$.
Hence, after $T$ we create a 3- or 5-cycle, or an ext-triangular 7-cycle, contradicting $(b)$.
Therefore, $D$ is good in $G'$.

Next let $T$ consist of deleting a nonempty set of internal vertices and adding an edge $e$ between two nonadjacent vertices.
Similarly, to complete the proof in this case, it suffices to guarantee that $G'$ has no special 9-cycles and that $D$ is good in $G'$.

Suppose $G'$ has a special 9-cycle $C$. Let $H$ be a bad partition of $C$. We have $e\in E(H)$ since otherwise $C$ is a special 9-cycle of $G$.
Hence, every cell of $H$ containing $e$ is created, which implies that we have created a 3- or 5-cycle or an ext-triangular 8-cycle, contradicting $(b)$.

Suppose $D$ is bad in $G'$. Let $H$ be a bad partition of $Int(D)$. Similarly, one can conclude that every cell of $H$ containing $e$ is created. Since $e\notin E(D)$ and $G' \in \cal{G}$, we create a 3- or 5-cycle or an ext-triangular 7-cycle, a contradiction.

At last, let $T$ consist of deleting all vertices in $S$ and identifying two edges $u_1u_2$ and $v_1v_2$.
Denote by $w_1$ the vertex of $G'$ obtained from $u_1$ and $v_1$ after $T$, and by $w_2$ one obtained from $u_2$ and $v_2$.
Since condition $(a)$ holds for both $T_1$ and $T_2$, $D$ bounds $G'$ and $\phi$ is a proper 3-coloring of $G'[V(D)]$.

Suppose we create a $6^-$-cycle $C'$ after $T$. Since condition $(b)$ holds for both $T_1$ and $T_2$,
we have $w_1,w_2\in V(C')$ and furthermore, one of the two paths of $C'$ between $w_1$ and $w_2$ connects $u_1$ and $u_2$, and the other connects $v_1$ and $v_2$. Clearly, $w_1$ and $w_2$ are nonconsecutive on $C'$, since otherwise $C'$ is a $6^-$-cycle of $G$. It follows that both $u_1u_2$ and $v_1v_2$ are contained in a $5^-$-cycles of $G-S$, a contradiction. Therefore, we create no $6^-$-cycle by $T$.
Furthermore, by a similar argument, one can conclude that we create no ext-triangular 7- or 8-cycle by $T$.

Suppose we create a special 9-cycle $C$ after $T$. Let $H$ be a bad partition of $C$.
Clearly, no cell of $H$ is created by $T$. It follows that $G$ has a 2-path between $u_1$ and $u_2$ and a 7-path between $v_1$ and $v_2$ so that edge $w_1w_2$ is a (3,8)-chord of $C$, since otherwise $C$ is a special 9-cycle of $G$. Now both $u_1u_2$ and $v_1v_2$ are contained in an $8^-$-cycle of $G$, a contradiction.
Therefore, we create no special 9-cycle after $T$.

Suppose $D$ is bad in $G'$. Let $H$ be a bad partition of $D$.
Notice that by $T$ we identify one pair of edges, and that each cell of $H$ has more than one edge shared with some other cell. If no cell of $H$ is created by $T$, then $D$ is bad in $G$.
Hence, we may assume that $H$ has a cell $C_H$ that is created by $T$.
Recall that condition $(b)$ holds for $T$, too. It follows that $H$ has either a (5,5,7)- or (5,5,8)-claw or a (5,5,5,7)-biclaw, and $C_H$ is the cell of length at least 7.
Furthermore, since $D$ is unchanged and no $6^-$-cycle is created after $T$, it is impossible that we create $C_H$ but no other cells of $H$ by $T$.
Therefore, $D$ is good in $G'$.

By the conclusions above, $\phi$ can be extended to $G'$ because of the minimality of $G$.
\end{proof}

\begin{lemma} \label{lem_good path}
Every face of $G$ contains no good path.
\end{lemma}

\begin{proof}
Suppose to the contrary that $G$ has a $k$-face $f$ that contains a good path $Q$. Since $G\in \cal{G}$, we have $k\geq 7$.
Let $f=[v_1\ldots v_k]$ and $Q=v_2\ldots v_5$. Let $t$ be a common neighbor of $v_2$ and $v_3$ not on $Q$, and $x$ be a neighbor of $v_4$ other than $v_3$ and $v_5$. Clearly, $x\neq v_1$. We do a graph operation $T$ on $G$ as follows: delete all vertices on $Q$ and identify $v_1$ and $x$, obtaining a plane graph $G'$.

Suppose that through $T$ we identify two vertices on
$D$, or create an edge connecting two vertices on $D$.
$G$ has a splitting 4- or 5-path $P$ of $D$ that contains
path $v_1\ldots v_4x$. Thus by Lemma \ref{lem_splitting path}, $G$ has a $9^-$-cycle $C$ formed by $P$ and $D$.
Clearly, $C$ is a good cycle and thus none of $t$ and $v_5$ lies inside $C$, which implies $t$ lies on $C$.
Now $C$ has two chords $tv_2$ and $tv_3$, a contradiction with Lemma \ref{lem_cycle of G}. Therefore, item $(a)$ in Lemma \ref{operation} holds for $T$.

Suppose that through $T$ we create a $6^-$-cycle or an ext-triangular 7- or 8-cycle.
Thus $G-v_5$ has a $12^-$-cycle $C$ containing path $v_1\ldots v_4x$, such that $\overline{Ext}(C)$ has a triangle adjacent to $C$ with common edge on $C-\{v_2,v_3,v_4\}$ when $|C|\in\{11,12\}$.
%Then $G-Q$ has a $6^-$-path or a ext-triangular 7- or 8-path $P$ between $v_1$ and $x$.
%Let $C$ be the cycle formed by paths $P$ and $v_1v_2v_3v_4x$. Clearly, $|C|\leq 12.$
It follows that $t\notin V(C)$ since otherwise $G$ has a $6^-$-face adjacent to triangle $[tv_2v_3]$.
Hence, $C$ is a bad cycle containing either $t$ or $v_5$ inside.
Now $C$ is adjacent to two triangles, contradicting Lemma \ref{lem triangular bad cycle}.
Therefore, item $(b)$ in Lemma \ref{operation} holds for $T$.

%If $t\notin C$, then $G$ has a $6^-$-cycle adjacent to triangle $[tv_2v_3t]$.
%If $t\in int(C)$, then $C$ is a bad cycle adjacent to a triangle both inside and outside.
%If $t\in ext(C)$, then $v_5\in int(C)$ so that $C$ is a bad cycle adjacent to two triangles outside.
%In each case, a contradiction is obtained.

%If $C$ is a good cycle, then none of $t, v_5$ lies inside $C$, which implies that $t\in C$.
%But then we have $|C|\geq 13$ since both the faces containing edges $tv_2$ and $tv_3$ other than triangle $[tv_2v_3t]$ has length at least 7, a contradiction.
%Hence, $C$ is a bad cycle and thus $P$ is a triangular 7- or 8-path.

%Clearly, $v_5$ lies either inside $C$ or outside $C$. If $v_5\in int(C)$, then $C$ is adjacent to two triangles outside.
%If $v_5\in ext(C)$, then $C$ is adjacent to a triangle outside and meanwhile has a 3-cell inside.
%In both cases, there is a contradiction with remark 2(2).
Hence $\phi$ can be extended to $G'$ by Lemma \ref{operation}. Next we extend $\phi$ from $G'$ to $G$: first properly color $v_5$ and $v_4$ in turn, then $v_2$ and $v_3$ can be properly colored since $v_1$ and $v_4$ receive different colors.
\end{proof}

\begin{lemma} \label{lem_face all 3}
$G$ has no $k$-face containing $k$ internal 3-vertex, where $k\in \{5,7\}$.
%\begin{enumerate}
%  \item $G$ has no completely internal 5-face with five incident 3-vertices.
%  \item $G$ has no completely internal 7-face with seven incident 3-vertices.
%\end{enumerate}
\end{lemma}

\begin{proof}
Suppose to the contrary that $G$ has such a $k$-face $f$. Let $f=[v_1\ldots v_k]$.
For $1\leq i\leq k$, denote by $u_i$ the neighbor of $v_i$ not on $f$.
Clearly, vertices $u_1,\cdots,u_k$ are pairwise distinct.

(1)~ Let $k=5$.
Since $G$ has no special 9-cycles, $f$ has a vertex incident with two 7$^+$-faces.
Without loss of generality, let $v_1$ be such a vertex.
We do a graph operation $T$ on $G$ as follows: delete all the vertices on $f$ and insert an edge between $u_5$ and $u_2$. Denote by $G'$ the resulting plane graph.

Suppose that $u_2,u_5\in V(D)$.
As a splitting 4-path of $D$, path $u_5v_5v_1v_2u_2$ together with $D$ forms a $5$- or $7$-face of $G$, an obvious contradiction. Therefore, item $(a)$ holds for $T$.

Suppose through $T$ we create a $6^-$-cycle or an ext-triangular 7- or 8-cycle.
Then $G-\{v_3,v_4\}$ has a $11^-$-cycle $C$ containing path $u_5v_5v_1v_2u_2$ such that $\overline{Ext}(C)$ has a triangle adjacent to $C$ with common edge on $C-\{v_5,v_1,v_2\}$ when $|C|\in \{10,11\}$.
%Suppose through $T$ we create a $6^-$-cycle or a ext-triangular 7- or 8-cycle.
%Then $G$ has a $5^-$-path or a ext-triangular 6- or 7-path, say $P$, between $u_5$ and $u_2$ such that $P$ contains no vertex on $f$.
%Let $C$ be the cycle formed by paths $P$ and $u_5v_5v_1v_2u_2$. Clearly, $C$ has length at most 11.
If $C$ is a good cycle, then none of $u_1,v_3$ and $v_4$ lies inside $C$, which implies $u_1\in V(C)$.
Now $u_1v_1$ divides $C$ into two cycles $C_1$ and $C_2$.
On one hand, since $v_1$ is incident with two 7$^+$-faces, we have $|C_1|, |C_2|\geq 7$.
On the other hand, we have $|C_1|+|C_2|=|C|+2\leq 13$. An contradiction is obtained.
Hence, we may assume $C$ is a bad 11-cycle.
It follows that $C$ has a $(5,5,7)$-claw by Lemma \ref{lem triangular bad cycle}, which is impossible since now either $C$ contains two vertices $v_3$ and $v_4$ inside or $\overline{Int}(C)$ has two 7$^+$-faces.
Therefore, item $(b)$ holds for $T$.

Hence by Lemma \ref{operation}, $\phi$ can be extended to $G'$. Notice that $u_1$ receives a color different from at least one of $u_2$ and $u_5$. Without loss of generality, let us say $u_2$.
We extend $\phi$ from $G'$ to $G$ in following way: color $v_2$ same as $u_1$, then $v_3,v_4,v_5$ and $v_1$ can be properly colored in turn.

(2)~ Let $k=7$. We do following operation $T$ on $G$: delete all vertices on $f$ and insert an edge between $u_1$ and $u_5$, obtaining a plane graph $G'$.

Suppose both $u_1$ and $u_5$ belong to $D$. Let $P=u_1v_1v_7v_6v_5u_5$. Since $P$ is a splitting path of $D$, $G$ has a $9^-$-cycle $C$ formed by $P$ and $D$ by Lemma \ref{lem_splitting path}. Clearly, $C$ is good. Thus $u_6,u_7\in V(C)$. Now $C$ has two chords, a contradiction with Lemma \ref{lem_cycle of G}. Therefore, item $(a)$ holds for $T$.

Suppose through $T$ we create a $6^-$-cycle or an ext-triangular 7- or 8-cycle.
Then $G-\{v_2,v_3,v_4\}$ has a $12^-$-cycle $C$ containing path $P$ such that $\overline{Ext}(C)$ has a triangle adjacent to $C$ with common edge on $C-\{v_1,v_7,v_6,v_5\}$ when $|C|\in \{11,12\}$.
If $C$ is a good cycle, then both $u_6$ and $u_7$ lie on $C$.
Since $|C|\leq12$, edges $v_6u_6$ and $v_7u_7$ divide $C$ into three cycles, each of which has length 5.
It follows that $|C|=11$ and hence $\overline{Int}(C)$ has a 5-face adjacent to a triangle, a contradiction.
Hence, we may assume $C$ is a bad cycle.
By Lemma \ref{lem triangular bad cycle}, $C$ has either a $(5,5,7)$-claw or a $(5,5,5,7)$-biclaw, which is impossible obviously.
Therefore, item $(b)$ holds for $T$.

\begin{comment}
Suppose through $T$ we create a $6^-$-cycle or a triangular 7- or 8-cycle. Then $G$ has a $5^-$-path or a triangular 6- or 7-path, say $P$, between $u_1$ and $u_5$ such that $P$ contains no deleted vertex.
Let $C$ be the cycle formed by paths $P$ and $u_1v_1v_7v_6v_5u_5$. Clearly, $C$ has length at most 12.
If $C$ is a good cycle, then both $u_6, u_7$ can only lie on $C$.
This yields that $|C| =11$ and furthermore, every face incident with one of edges $v_1v_7,v_7v_6$ and $v_6v_5$ other than $f$ has length 5.
Thus $P$ is a triangular 6-path, which is impossible since otherwise $G$ has a 6-cycle with a (3,5)-chord.
Hence $C$ is a bad cycle and thus $P$ is a triangular 6- or 7-path. It follows that $C$ has either a $(5,5,7)$-claw or a (5,5,5,7)-biclaw by remark 2(2).
This is impossible since now $C$ either contains vertices $v_2,v_3,v_4$ inside, or $u_6,u_7$ inside whose adjacency implies a 4-cycle $v_7v_6u_6u_7v_7$ of $G$.
\end{comment}
Hence by Lemma \ref{operation}, $\phi$ can be extended to $G'$.
Furthermore, $\phi$ can be extended from $G'$ to $G$ in a similar way as part (1) of this lemma.
\end{proof}

\begin{lemma} \label{lem_two light 7-face}
$G$ has no two 7-faces $[xv_1\ldots v_6]$ and $[xu_1\ldots u_6]$ such that $x$ is their unique common vertex, $u_1$ and $v_1$ are adjacent, both $x$ and $u_1$ are internal 4-vertices, and all other vertices on these two 7-faces are internal 3-vertices.
\end{lemma}

\begin{proof}
Suppose to the contrary $G$ has such two 7-faces. Let $f=[xv_1\ldots v_6]$ and $g=[xu_1\ldots u_6]$. Let $y$ and $z$ be the neighbors of $u_1$ and $v_6$ not on $f\cup g$, respectively.
We do the following operation $T$ on $G$: delete both $V(f)$ and $V(g)$, and identify $z$ and $y$, obtaining a plane graph $G'$.

Suppose through $T$ we identify two vertices on
$D$, or create an edge connecting two vertices on $D$.
Then $G$ has a splitting 4- or 5-path $P$ of $D$ containing
path $yu_1xv_6z$. It follows from Lemma \ref{lem_splitting path} that $G$ has a $9^-$-cycle $C$ formed by $P$ and $D$.
Hence, $C$ is a good cycle and thus not separating, contradicting that $C$ has either $u_2$ or $v_1$ inside.
Therefore, item $(a)$ holds for $T$.

Suppose through $T$ we create a $6^-$-cycle or an ext-triangular 7- or 8-cycle.
Then $G-V(f)\cup V(g)$ has a $8^-$-path between $y$ and $z$, which together with path $yu_1xv_6z$ form a $12^-$-cycle $C$.
It follows that $G$ has at most three vertices inside $C$, contradicting the fact that now either $u_2,\ldots, u_6$ or $v_1,\ldots,v_5$ lie inside $C$.
Therefore, item $(b)$ holds for $T$.

Hence by Lemma \ref{operation}, $\phi$ can be extended to $G'$.
We further extend $\phi$ from $G'$ to $G$ in following way: first color $x$ same as $y$, then $u_6,\ldots,u_1$ can be properly colored in turn, and so do $v_1,\ldots,v_6$.
\end{proof}

\begin{lemma} \label{lem_theta}
$G$ has no 8-cycle $[xyzu_1\ldots u_5]$ with a chord $xz$ such that $z$ is an internal 4-vertex and all other vertices of this 8-cycle are internal 3-vertices.
\end{lemma}

\begin{proof}
Suppose to the contrary that $G$ has such an 8-cycle $C$.
Let $z'$ and $y'$ be the neighbors of $z$ and $y$ not on $C$, respectively.
We remove $C$ from $G$ to obtain a plane graph $G'$ with fewer vertices.
By the minimality of $G$, $\phi$ can be extended to $G'$. We complete the proof by extending $\phi$ from $G'$ to $G$ in following way: if $z'$ and $y'$ receive a same color, then we color $x$ same as $z'$ and finally, $u_5,\ldots,u_1,z,y$ can be properly colored in turn; otherwise, we color $z$ same as $y'$, and then $u_1,\ldots,u_5,x,y$ can be properly colored in turn.
\end{proof}

\begin{lemma} \label{lem_9-face}
$G$ has no 9-face $[u_1\ldots u_9]$ such that $u_1,u_2,u_3,u_5,u_6,u_7$ are six bad vertices and $u_4$ is a 4-vertex incident with two 3-faces.
\end{lemma}

\begin{proof}
Suppose to the contrary $G$ has such a 9-face $f$.
$G$ has four 3-faces $[xu_1u_2]$, $[yu_3u_4]$, $[zu_4u_5]$, $[wu_6u_7]$ adjacent to $f$.
%Clearly, all the edges $v_1v_2, v_3v_4,v_4v_5,v_6v_7$ are triangular. Let $x, y, z, w$ be the common neighbors of $v_1$ %and $v_2$, $v_3$ and $v_4$, $v_4$ and $v_5$, $v_6$ and $v_7$, respectively.
Let $S=\{u_1,u_2,u_3,u_5,u_6,u_7\}$.
We apply following graph operation $T$ on $G$ to obtain a plane graph $G'$ with fewer vertices: delete all vertices of $S$ and identify two edges $u_8u_9$ and $zu_4$ so that $u_8$ is identified with $z$.
Denote by $T_1$ (or $T_2$) the graph operation on $G$ that consists of deleting all vertices in $S$ and identifying $u_8$ and $z$ (or $u_9$ and $u_4$).
Similarly as the proof of Lemma \ref{lem_good path}, one can conclude that items $(a)$ and $(b)$ hold for both $T_1$ and $T_2$.
Besides, $u_4z$ is contained in no $8^-$-cycle of $G-S$. Hence, $\phi$ can be extended to $G'$ by Lemma \ref{operation}.
Furthermore, we can extend $\phi$ from $G'$ to $G$ in a similar way as Lemma \ref{lem_good path}.

%Furthermore, we extend $\phi$ from $G'$ to $G$ as follows:
%color $u_5$ different from $u_4$ and $z$, then $u_6,u_7$ can be properly colored since $u_5,u_8$ receive different colors; similarly, $u_1, u_2,u_3$ can be properly colored.
\end{proof}

\begin{comment}
\newtheorem{lem5}[lemsure1]{Lemma}
\begin{lem5}
$G$ has no 9-face $f=[u_1u_2\cdots u_9]$ such that $f$ contains five triangular edges $u_1u_2$, $u_2u_3$, $u_4u_5$, $u_6u_7$, $u_8u_9$, and all the vertices incident with $f$ have degree 3 other than $v_2$, $v_6, v_7$, each of which has degree 4.
\end{lem5}

\newtheorem{lem6}[lemsure1]{Lemma}
\begin{lem6}
$G$ has no completely internal 9-face $f=[u_1,u_2,\cdots,u_9]$ such that $f$ contains five triangular edges $u_2u_3, u_4u_5, u_6u_7, u_7u_8, u_9u_1$, and all the vertices incident with $f$ have degree 3 other than $u_2, u_7$, both of which has degree 4.
\end{lem6}
\end{comment}

\subsection{Discharging in $G$}
\label{secch}
Let $V=V(G)$, $E=E(G)$, and $F$ be the set of faces of $G$.
Denote by $f_0$ the exterior face of $G$.
Give initial charge $ch(x)$ to each element $x$ of $V\cup F$, where $ch(f_0)=d(f_0)+4$, $ch(v)=d(v)-4$ for $v\in V$, and $ch(f)=|f|-4$ for $f\in F\setminus \{f_0\}$.
Discharge the elements of $V\cup F$ according to the following rules:

\begin{enumerate}[$R1.$]
  \item Every 3-face receives $\frac{1}{3}$ from each incident vertex.
  \item Let $v$ be an internal 3-vertex and $f$ be a face containing $v$.
  \begin{enumerate}[(1)]
    \item Vertex $v$ receives $\frac{1}{4}$ from $f$ if $d(f)=5$.
    \item Suppose $d(f)\geq7$. Let $a$ and $b$ denote the lengths of two faces containing $v$ other than $f$, and $a\leq b$. Vertex $v$ receives from $f$ charge $\frac{2}{3}$ if $a=3$, charge $\frac{1}{2}$ if $a=b=5$, charge $\frac{3}{8}$ if $a=5$ and $b\geq7$, and charge $\frac{1}{3}$ if $a\geq 7$.
  \end{enumerate}
  \item Let $v$ be an internal 4-vertex and $f$ be a $7^+$-face containing $v$.
  \begin{enumerate}[(1)]
    \item If $v$ is incident with precisely two 3-faces, then $v$ receives $\frac{1}{3}$ from $f$.
    \item If $v$ is incident with precisely one 3-face that is adjacent to $f$, then $v$ receives $\frac{1}{6}$ from $f$.
  \end{enumerate}
    \item Let $f$ be a light 7-face adjacent to a 3-face $T$ on edge $xy$, $z$ be the vertex on $T$ other than $x$ and $y$, and $h$ be the face containing edge $yz$ other than $T$.
  \begin{enumerate}[(1)]
    \item If $d(x)=3$ and $d(y)\geq 5$, then $y$ sends $\frac{1}{24}$ to $f$.
    \item If $z\in V(D)$, then $z$ sends $\frac{5}{24}$ to $f$ through $T$.
    \item If $d(x)=3, d(y)=4, z\notin V(D)$ and $d(z)\geq4$, then $h$ sends $\frac{5}{24}$ to $f$ through $y$.
  \end{enumerate}
  \item The exterior face $f_0$ sends $\frac{4}{3}$ to each incident vertex.
  \item Let $v$ be an external vertex and $f$ be a $5^+$-face containing $v$ other than $f_0$.
  \begin{enumerate}[(1)]
    \item If $d(v)=2$, then $v$ receives $\frac{2}{3}$ from $f$.
    \item Suppose $d(v)=3$. If $v$ is triangular, then $v$ receives $\frac{1}{12}$ from $f$; otherwise, $v$ sends $\frac{1}{12}$ to $f$.
    \item If $d(v)\geq4$, then $v$ sends $\frac{1}{3}$ to $f$.
  \end{enumerate}
\end{enumerate}

Let $ch^*(x)$ denote the final charge of each element $x$ of $V\cup F$ after discharging.
On one hand, by Euler's formula we deduce $\sum\limits_{x\in V\cup F}ch(x)=0.$
Since the sum of charge over all elements of $V\cup F$ is unchanged, we have $\sum\limits_{x\in V\cup F}ch^*(x)=0.$ On the other hand, we show that $ch^*(x)\geq 0$ for $x\in V\cup F$ and  $ch^*(x_0)> 0$ for some vertex $x_0$. Hence, this obvious contradiction completes the proof of Theorem \ref{thm46s9}.

It remains to show that $ch^*(x)\geq 0$ for $x\in V\cup F$ and  $ch^*(x_0)> 0$ for some vertex $x_0$.
\begin{claim}
$ch^*(f)\geq0$ for $f\in F$.
\end{claim}

Denote by $V(f)$ the set of vertices of $f$.

First suppose that $f$ contains no external vertex.

Let $|f|=3$. By $R1$, we have $ch^{*}(f)=|f|-4+3\times \frac{1}{3}=0$.

Let $|f|=5$. Lemma \ref{lem_face all 3} implies that $f$ contains at most four 3-vertices. Hence, we have $ch^*(f)\geq|f|-4-4\times \frac{1}{4}=0$ by $R2(1)$.

Let $|f|=7$. If $G$ has no 3-face adjacent to $f$, then $f$ sends at most $\frac{1}{2}$ to each incident 3-vertex by $R2(2)$. Since Lemma \ref{lem_face all 3} implies that $f$ contains at most six 3-vertices, we have $ch^*(f)\geq|f|-4-6\times \frac{1}{2}=0.$ Hence, we may assume that $f$ is adjacent to a 3-face $T=[xyz]$ on edge $xy$, where $d(x)\leq d(y).$ Since $G$ has no special 9-cycle, $f$ is adjacent to no other 3-face than $T$.
Notice that now only rules $R2(2), R3(2)$ and $R4(3)$ might make $f$ send charge out.

%Remind that $f$ sends at most $\frac{1}{2}$ to each of incident vertices other than $x$ and $y$ by remark 2.
Suppose $d(y)=3$. In this case $f$ sends $\frac{2}{3}$ to both $x$ and $y$, and at most $\frac{1}{2}$ to each of other incident 3-vertices.
Moreover, it follows from Lemma \ref{lem_good path} that $f$ contains at least two $4^+$-vertices. Hence, we have $ch^*(f)\geq |f|-4-2\times \frac{2}{3}-3\times \frac{1}{2}>0.$

%Case1: suppose that $\{d(x), d(y)\}=\{3, 3\}$. Then by remark 2, $f$ sends $\frac{2}{3}$ to both $x$ and $y$. Moreover, it follows from lemma 7 that $f$ has at least two $4^+$-vertices, each of which receives nothing from $f$. Hence, we have $ch^*(f) >7-4-2\times \frac{2}{3}-3\times \frac{1}{2}>0.$

Suppose $d(x)=3$ and $d(y)=4$. In this case $f$ sends $\frac{2}{3}$ to $x$, at most $\frac{1}{6}$ to $y$, and at most $\frac{3}{8}$ to the neighbor of $x$ on $f$ other than $y$.
If $z$ is not an internal 3-vertex, then $f$ receives charge $\frac{5}{24}$ either from $z$ by $R6(3)$ or from the face containing $yz$ other than $T$ by $R4(3)$, yielding $ch^*(f)\geq |f|-4-\frac{2}{3}-\frac{1}{6}-\frac{3}{8}-4\times \frac{1}{2}+\frac{5}{24}=0$.
Hence, we may assume $z$ is an internal 3-vertex.
Since Lemma \ref{lem_theta} implies $f$ is not light, we have $ch^*(f)\geq |f|-4-\frac{2}{3}-\frac{1}{6}-4\times \frac{1}{2}>0$.

%Case2: suppose that $\{d(x), d(y)\}=\{4^+,4^+\}$. Let $f_{xz}, f_{yz}$ be the faces adjacent to $T$ with $xz, yz$ as their common edge, respectively. Then $y$ receives from $f$ $\frac{3}{8}$ if $R4$ works through $y$, at most $\frac{1}{6}$ otherwise. And so does $x$. If at least one of $x,y$ has degree more than 4, then we have $ch^*(f)\geq 7-4-\frac{3}{8}-5\times \frac{1}{2}> 0.$ Hence, we can assume that $d(x)=d(y)=4$. If $R4$ works neither through $y$ nor through $z$, we have $ch^*(f)\geq 7-4-2\times \frac{1}{6}-5\times \frac{1}{2}> 0.$ It remains to assume that $R4$ works through at least one of $y$ and $z$, say $y$. Thus $f_{yz}$ is a bad 7-face with $d(z)=3$ . Then by lemma 9 $f$ is not a bad 7-face, i.e., $f$ has at least one $4^+$-vertex receiving nothing from $f$. Hence, we have $ch^*(f)\geq 7-4-2\times \frac{3}{8}-4\times \frac{1}{2}> 0.$

It remains to suppose $d(x)\geq 4$. In this case, $f$ might send charge out through $x$ and $y$ by $R4(3)$.
If $f$ is not light, then $ch^*(f)\geq |f|-4-2(\frac{1}{6}+\frac{5}{24})-4\times \frac{1}{2}>0$.
If $d(y)\geq 5$, then $f$ sends nothing to $y$ or through $y$, yielding $ch^*(f)\geq |f|-4-(\frac{1}{6}+\frac{5}{24})-5\times \frac{1}{2}>0$.
Hence, we may assume that $f$ is light and $d(x)=d(y)=4$. Lemma \ref{lem_two light 7-face} implies that $f$ sends nothing out through $x$ or $y$.
It follows that $ch^*(f)\geq 7-4-2\times \frac{1}{6}-5\times \frac{1}{2}> 0.$

%Case3: suppose that $\{d(x), d(y)\}=\{3,4^+\}$. Without loss of generality, assume that $d(x)=3, d(y)\geq 4$. Denote by $u$ the neighbor of $x$ other than $y$ on $f$. Then $f$ sends $\frac{2}{3}$ to $x$, at most $\frac{1}{6}$ to $y$ and at most $\frac{3}{8}$ to $u$. If $f$ is not a bad 7-face, then we have $ch^*(f)\geq 7-4-\frac{2}{3}-\frac{1}{6}-4\times \frac{1}{2}>0$. Hence, we can assume that $f$ is a bad 7-face. If $d(y)\geq 5$, then $f$ receives $\frac{1}{24}$ from $y$ by $R4(2)$, yielding $ch^*(f)\geq 7-4-\frac{2}{3}-\frac{3}{8}-4\times \frac{1}{2}+\frac{1}{24}=0$. Hence, we can assume that $d(y)=4$. By lemma 10, $z$ is not an internal 3-vertex, which implies either $z\in D$ or $z$ is an internal $4^+$-vertex. If $z\in D$, then $f$ receives $\frac{5}{24}$ from $z$ by $R4(3)$. If $z$ is an internal $4^+$-vertex, then $f$ receives $\frac{5}{24}$ from the face containing edge $yz$ other than $T$ by $R4(1)$. In both case we have $ch^*(f)\geq 7-4-\frac{2}{3}-\frac{1}{6}-\frac{3}{8}-4\times \frac{1}{2}+\frac{5}{24}=0$.

Let $|f|=8$. Since $f$ sends at most $\frac{1}{2}$ to each incident vertex by $R2(2)$ , we have $ch^*(f)\geq 8-4-8\times \frac{1}{2}=0$.

Let $|f|\geq 9$. We define
\begin{align*}
&A(f)=\{v\colon\ \textit{uvw is a path on f, both u and w are bad, and v is good}\},\\
&B(f)=\{v\colon\ \textit{uvw is a path on f, u is bad, and both v and w are good}\},\\
&C(f)=\{v\colon\ \textit{uvw is a path on f, and all of u, v and w are good}\},\\
&D(f)=\{v\colon\ \textit{v is a bad vertex on f}\}.
\end{align*}
Clearly, $A(f),B(f),C(f)$ and $D(f)$ are pairwise disjoint sets whose union is $V(f)$.
By our rules, $f$ sends at most $\frac{1}{3}$ to each vertex in $A(f)$,
at most $\frac{3}{8}$ in total to and through each vertex in $B(f)$,
at most $\frac{1}{2}$ in total to and through each vertex in $C(f)$ and
$\frac{2}{3}$ to each vertex in $D(f)$.
Hence, we have
\begin{align*}
ch^*(f)&\geq |f|-4-\frac{1}{3}|A(f)|-\frac{3}{8}|B(f)|-\frac{1}{2}|C(f)|-\frac{2}{3}(|f|-|A(f)|-|B(f)|-|C(f)|)\\
&=\frac{1}{3}|A(f)|+\frac{7}{24}|B(f)|+\frac{1}{6}|C(f)|+\frac{1}{3}|f|-4.\tag{$\ast$}
\end{align*}
Clearly, $|B(f)|$ is always even, and if $B(f)=\emptyset$, then either $C(f)=\emptyset$ or $C(f)=V(f)$.
Also note that $f$ sends nothing through a vertex $u$ of $f$ if $f$ has a vertex $v$ such that $uv$ is a common edge of $f$ and a 3-face of $G$.

Suppose $|f|=9$. By inequality ($\ast$), it suffices to consider following three cases.

Case 1: $|A(f)|\leq 2$ and $|B(f)|=|C(f)|=0$. By Lemma \ref{lem_good path}, we have $|A(f)|=2$ (say $A(f)=\{u,v\}$), $D(f)$ is divided by $u$ and $v$ as 3+4 on $f$, and $d(u),d(v)\geq 4$. Through the drawing of 3-faces adjacent to $f$, one can find that Lemma \ref{lem_9-face} implies that not both $u$ and $v$ have degree 4. Hence, we have $ch^*(f)\geq |f|-4-7\times \frac{2}{3}-\frac{1}{3}=0.$

Case 2: $|A(f)|=1, |B(f)|=2$ and $|C(f)|=0$. By Lemma \ref{lem_good path}, $D(f)$ is divided by $B(f)\cup A(f)$ as 3+3 or 2+4 on $f$.

In the former case 3+3, let $A(f)=\{u\}$. By Lemma \ref{lem_9-face}, $u$ is not a 4-vertex incident with two 3-faces, and thus receives at most $\frac{1}{6}$ from $f$. Hence, we have $ch^*(f)\geq |f|-4-6\times \frac{2}{3}-2\times \frac{3}{8}-\frac{1}{6}>0.$

In the latter case 2+4, let $f=[u_1\ldots u_9]$, $u_1\in A(f)$, and $u_4,u_5\in B(f)$. Lemma \ref{lem_good path} implies $d(u_1),d(u_5)\geq 4$.
Furthermore, $u_1$ is a 4-vertex incident with two 3-faces, since otherwise $f$ sends at most $\frac{1}{6}$ to $u_1$ so that $ch^*(f)\geq |f|-4-6\times \frac{2}{3}-2\times \frac{3}{8}-\frac{1}{6}>0.$ Through the drawing of 3-faces adjacent to $f$, one can find that $d(u_4)\geq 4$. Hence, $f$ sends nothing through $u_4$ and $u_5$, and at most $\frac{1}{3}$ to each of them, yielding $ch^*(f)\geq |f|-4-6\times \frac{2}{3}-3\times \frac{1}{3}=0.$

Case 3: $|A(f)|=0, |B(f)|=2$ and $|C(f)|\leq 2$. It follows that $f$ contains five consecutive bad vertices, and hence has a good path, contradicting Lemma \ref{lem_good path}.

Suppose $|f|\geq 10$.
If $|A(f)|+\frac{|B(f)|}{2}\geq 2$, then by inequality ($\ast$) we are done.
Hence, we may assume either $|A(f)|\leq 1$ and $|B(f)|=0$, or $|A(f)|=0$ and $|B(f)|=2$.
Lemma \ref{lem_good path} implies a contradiction in the former case, and $|C(f)|\geq 4$ in the latter case.
Hence, by inequality ($\ast$) we are also done in the latter case.

Next suppose $f$ contains external vertices. Since $|f_0|\leq 12$, if $f=f_0$ then by $R5$ we have $ch^*(f)=|f_0|+4-|f_0|\times \frac{4}{3}\geq0.$
Hence, we may assume $f\neq f_0.$ By our rules, $f$ sends at most $\frac{2}{3}$ to each incident vertex.
Lemma \ref{lem_splitting path} implies that if $|f|\leq 8$, then the external vertices on $f$ are consecutive one by one. Furthermore, $f$ has at most one 2-vertex if $|f|=5$, and has at most two 2-vertices if $|f|\in \{7, 8\}$.

Let $|f|=3$. We have $ch^{*}(f)=|f|-4+3\times \frac{1}{3}=0$ by $R1$.

Let $|f|=5$. If $f$ has no 2-vertex, then $f$ sends at most $\frac{1}{4}$ to each vertex, yielding $ch^*(f)\geq |f|-4-4\times \frac{1}{4}=0$.
Hence, we may assume $f$ has precisely one $2$-vertex. It follows that $f$ has two external 3-vertices, both of which send at least $\frac{1}{12}$ to $f$ by $R6$. Hence, we have $ch^*(f)\geq |f|-4-\frac{2}{3}+2\times \frac{1}{12}-2\times \frac{1}{4}=0$.

Let $|f|=7$. Since in this case $f$ is adjacent to at most one 3-face, $f$ has an internal vertex that is not bad. By our rules, $f$ sends at most $\frac{1}{2}$ to this vertex.
If $f$ has an external $4^+$-vertex, then $f$ receives $\frac{1}{3}$ from this vertex by $R6(3)$, yielding $ch^*(f)\geq |f|-4+\frac{1}{3}-4\times \frac{2}{3}-\frac{1}{2}>0$.
Hence, we may assume that $f$ has no external $4^+$-vertex, which implies $f$ has two external 3-vertices $u$ and $v$.
If both of $u$ and $v$ are not triangular and thus send $\frac{1}{12}$ to $f$, then we have $ch^*(f)\geq |f|-4+2\times \frac{1}{12}-4\times \frac{2}{3}-\frac{1}{2}=0$.
Hence, we may assume that $u$ is triangular but $v$ not. Now $f$ has at most one bad vertex, yielding $ch^*(f)\geq |f|-4+\frac{1}{12}-\frac{1}{12}-3\times \frac{2}{3}-2\times \frac{1}{2}=0$.

Let $|f|=8$.
If $f$ has no 2-vertex, then $f$ sends at most $\frac{1}{2}$ to each incident vertex, yielding $ch^*(f)\geq |f|-4-8\times \frac{1}{2}=0$.
Hence, we may assume that $f$ has precisely one or two $2$-vertices. It follows that $f$ has two external $3^+$-vertices, both of which send at least $\frac{1}{12}$ to $f$. Hence we have $ch^*(f)\geq |f|-4-2\times \frac{2}{3}+2\times \frac{1}{12}-4\times \frac{1}{2}>0$.

It remains to suppose $|f|\geq 9$. If $f$ has an external $4^+$-vertex, then $f$ receives $\frac{1}{3}$ from this vertex by $R6(3)$, yielding $ch^*(f)\geq |f|-4+\frac{1}{3}-(|f|-1)\times \frac{2}{3}\geq 0$.
Hence, we may assume that $f$ has no external $4^+$-vertex, which implies $f$ has at least two external 3-vertices. By $R6$, we have $ch^*(f)\geq |f|-4-2\times \frac{1}{12}-(|f|-2)\times \frac{2}{3}> 0$.

\begin{claim}
$ch^{*}(v)\geq0$ for $v\in V$.
\end{claim}

First suppose that $v$ is internal. We have $d(v)\geq3$ by Lemma \ref{lem_min degree}.

Let $d(v)=3$. Since $G\in \cal{G}$, the set of lengths of the faces containing $v$ is one of the followings: $\{3,7^+,7^+\},\{5,5,7^+\},\{5,7^+,7^+\}$ and $\{7^+,7^+,7^+\}$. Hence, we are done in each case by $R1$ and $R2$.

If $d(v)=4,$ then by $R1$ and $R3$ the charge $v$ sends out equals to the charge $v$ receives, yielding that $ch^*(v)=d(v)-4=0.$

It remains to suppose $d(v)\geq 5$. By $R1$ and $R4(1)$, $v$ sends $\frac{1}{3}$ to each incident 3-face and at most $\frac{1}{24}$ to each other incident face, which gives $ch^*(v)> d(v)-4-\frac{d(v)}{2}\times \frac{1}{3}-\frac{d(v)}{2}\times \frac{1}{24}>0$.

Next suppose that $v$ is external. Clearly, $d(v)\geq 2$.

By $R1$, $R5$ and $R6$, we have $ch^*(v)=d(v)-4+\frac{4}{3}+\frac{2}{3}=0$ if $d(v)=2$, $ch^*(v)=d(v)-4+\frac{4}{3}-\frac{1}{3}+\frac{1}{12}>0$ if $d(v)=3$ and $v$ is triangular, and $ch^*(v)=d(v)-4+\frac{4}{3}-\frac{1}{12}-\frac{1}{12}>0$ if $d(v)=3$ and $v$ is not triangular.

It remains to suppose $d(v)\geq 4$. Then $v$ receives $\frac{4}{3}$ from $f_0$ by $R5$, sends $\frac{1}{3}$ to each other incident face by $R1$ and $R6(3)$, and might sends $\frac{5}{24}$ through each incident 3-face whose other two vertices are internal. It follows that $ch^*(v)\geq d(v)-4+\frac{4}{3}-(d(v)-1)\times \frac{1}{3}-\frac{d(v)-2}{2}\times \frac{5}{24}>0$.

\begin{claim}
$D$ contains a vertex $x_0$ such that $ch^*(x_0) >0.$
\end{claim}

Let $x_0$ be any $3^+$-vertex on $D$, as desired.

The proof of Theorem \ref{thm46s9} is completed.

\end{document}